\documentclass[11pt]{amsart}

\usepackage{amssymb,amscd,amsthm,amsxtra}
\usepackage{dsfont,latexsym}

\usepackage{mathrsfs}

\usepackage{color}
\definecolor{citation}{rgb}{0.2,0.58,0.2} 
\definecolor{formula}{rgb}{0.1,0.2,0.6}
\definecolor{brown}{rgb}{0.4,0.1,0.1}
\definecolor{url}{rgb}{0.3,0,0.5} 

\usepackage[colorlinks=true,linkcolor=formula,urlcolor=url,citecolor=citation]{hyperref}

\newtheorem{thm}{Theorem}[section]

\newtheorem{lemma}[thm]{Lemma}

\theoremstyle{definition}
\newtheorem{defn}[thm]{Definition}
\theoremstyle{remark}
\newtheorem{rem}[thm]{Remark}
\numberwithin{equation}{section}

\newcommand{\R}{{\mathds R}}

\newcommand{\Om}{{\Omega}}

\newcommand{\ff}{{\mathcal F}}

\newcommand{\ta}{{\text {\rm Tail}}}

\def\Xint#1{\mathchoice
      {\XXint\displaystyle\textstyle{#1}}%
      {\XXint\textstyle\scriptstyle{#1}}%
      {\XXint\scriptstyle\scriptscriptstyle{#1}}%
      {\XXint\scriptscriptstyle\scriptscriptstyle{#1}} %
\!\int}
   \def\XXint#1#2#3{{\setbox0=\hbox{$#1{#2#3}{\int}$}
        \vcenter{\hbox{$#2#3$}}\kern-.5\wd0}}
   
   \def\dashint{\Xint-}

\def\dys{\displaystyle}

\def\eps{\varepsilon}
\def\dxy{\,{\rm d}x{\rm d}y}
\def\dy{\,{\rm d}y}
\def\dx{\,{\rm d}x}
\def\dn{\,{\rm d}\nu}
\def\dnb{\,{\rm d}{\bar\nu}}
\def\vs{\vspace{1mm}}

\allowdisplaybreaks
\makeatletter
\DeclareRobustCommand*{\bfseries}{%
  \not@math@alphabet\bfseries\mathbf
  \fontseries\bfdefault\selectfont
  \boldmath
}
\makeatother

\newlength{\defbaselineskip}
\newcommand{\setlinespacing}[1]
           {\setlength{\baselineskip}{#1 \defbaselineskip}}

\title
[Nonlocal Harnack inequalities]
{Nonlocal Harnack inequalities}

\author[A. Di Castro]{Agnese Di Castro}
\email[Agnese Di Castro]{\href{mailto:dicastro@mail.dm.unipi.it}{dicastro@mail.dm.unipi.it}}
\author[T. Kuusi]{Tuomo Kuusi}
\email[Tuomo Kuusi]{\href{mailto:tuomo.kuusi@aalto.fi}{tuomo.kuusi@aalto.fi}}
\author[G. Palatucci]{Giampiero Palatucci}
\email[Giampiero Palatucci]{\href{mailto:giampiero.palatucci@unipr.it}{giampiero.palatucci@unipr.it}}

\address[A. Di Castro]
{Dipartimento di Matematica, Universit\`a di Pisa
\\ Largo Bruno Pontecorvo,~5
\\ 56127 Pisa, Italy}

\address[T. Kuusi]
{Department of Mathematics and Systems Analysis, Aalto University
\\ P.O. Box 1100
\\ 00076 Aalto, Finland
}

\address[G. Palatucci]
{Dipartimento di Matematica e Informatica, Universit\`a degli Studi di Parma
\\ Campus - Parco Area delle Scienze~53/A
\\ 43124 Parma, Italy; \newline
\and \newline SISSA 
 \\ Via Bonomea 256 \\ 34136 Trieste, Italy
}

\begin{document}


\keywords{Quasilinear nonlocal operators, fractional Sobolev spaces, H\"older regularity, Caccioppoli estimates, singular perturbations, Harnack inequality\vspace{0.5mm}}

\begin{abstract}
\small We state and prove a general Harnack inequality for minimizers of nonlocal, possibly degenerate, integro-differential operators, whose model is the fractional $p$-Laplacian.
\end{abstract}

\maketitle

\setcounter{tocdepth}{2}

\begin{center}
 \rule{11.8cm}{0.5pt}\\[-0.1cm] 
{\sc {\small To appear in}\, {\it J. Funct. Anal.}}
\\[-0.25cm] \rule{11.8cm}{0.5pt}
\end{center}
\vspace{4mm}

 \setlinespacing{1.1}

\section{Introduction}

In the present paper we 
 deal with an extended class of operators, which include, as a particular case, some fractional powers of the Laplacian. Precisely, let $\Om$ be a bounded domain and take $g$ in the fractional Sobolev spaces $W^{s,p}(\R^n)$, for any $s\in (0,1)$ and any $p>1$.
We will  
prove general Harnack inequalities  
for the weak solutions $u$ to the following class of integro-differential problems
\begin{equation}\label{problema}
\begin{cases}
\mathcal{L} u=0 & \text{in }\Omega,\\
u=g & \text{in }\mathds{R}^n\setminus\Omega,
\end{cases}
\end{equation}
where the operator $\mathcal{L}$ is defined by
\begin{equation}\label{lisa}
\mathcal{L}u(x)=P.~\!V.\!\int_{\mathds{R}^n} K_{\textrm{\tiny sym}}(x,y)|u(x)-u(y)|^{p-2}(u(x)-u(y))\,{\rm d}y, \quad x\in \mathds{R}^n;
\end{equation}
the symbol $P.~\!V.$ means ``in the principal value sense''; and 
$K$ is a suitable kernel of order $(s,p)$ with merely measurable coefficients. Above~$K_{\textrm{\tiny sym}}$ is the symmetric part of~$K$ defined as $K_{\textrm{\tiny sym}}(x,y) = (K(x,y)+K(y,x))/2$. 
Equivalently,  we will consider  the minimizers of the following class of nonlocal functionals 
\begin{equation}\label{funzionale}
\dys
\ff (v):= \int_{\R^n} \int_{\R^n} K(x,y)|v(x)-v(y)|^p \dxy,
\end{equation}
whose domain of definition is $v\in W^{s,p}(\R^n)$. Specifically, we shall consider minimization problems with prescribed boundary values, i.~\!e., $v=g$ on  $\R^n \setminus \Om$. These minimizers indeed coincide  with the solutions to~\eqref{problema}, as seen, e.~\!g., in Theorem 2.3 in \cite{DKP13}.
We refer  to Section~\ref{sec_preliminaries} for the precise assumptions on the involved quantities. However, in order to simplify, one can just keep in mind the model case when the kernel $K(x,y)$  coincides with $|x-y|^{-n-sp}$; that is, the function $u$ is the  solution to the following problem
\begin{equation*}\label{def_p1}
\begin{cases}
\dys
(-\Delta)^{s}_p \,u = 0 & \text{in} \ \Omega, \\[1ex]
u = g & \text{in} \ \R^n\setminus \Omega,
\end{cases}
\end{equation*}
where the symbol~$\dys (-\Delta)^{s}_p$ denotes the usual {\it fractional $p$-Laplacian} operator, though in such a case the difficulties arising from having merely measurable coefficients disappear.

To formulate our main results, there is a special quantity appearing in estimates and being fundamental when we deal with nonlocal operators. Namely, we define the \textit{nonlocal tail} of a function $v\in W^{s,p}(\mathds{R}^n)$ as
\begin{equation}\label{tail_int}
{\rm Tail }(v;{x_0},R):=\left[R^{sp}\int_{\mathds{R}^n \setminus B_R({x_0})}|v(x)|^{p-1}|x-{x_0}|^{-(n+sp)}\,{\rm d}x\right]^{\frac{1}{p-1}}.
\end{equation}
Note that the quantity above is finite whenever $v\in L^q(\R^n)$, $q\geq p-1$ and $R>0$. The definition already appears in~\cite{DKP13}. 
The way how the nonlocal tail will be handled is one of the key-points in the proof of the main result of our paper, which reads as follows
\begin{thm}[\bf Nonlocal Harnack inequality]\label{harnack}
For any $s\in (0,1)$ and any $p \in (1,\infty)$, let $u\in W^{s,p}(\mathds{R}^n)$ be a weak solution to~\eqref{problema} such that $u\geq 0$ in $B_R\equiv B_R({x_0})\subset\Omega$. Then the following estimate holds for any $B_r\equiv B_r({x_0})\subset B_{R/2}(x_0)$,
\begin{equation}\label{cloe}
\sup_{B_r} u \, \leq \, c \inf_{B_r} u + c \left(\frac{r}{R}\right)^{\frac{sp}{p-1}}\text{\rm Tail}(u_{-}; {x_0},R),
\end{equation}
where ${\rm Tail}(\cdot)$ is defined in \eqref{tail_int}, $u_-=\max\{-u, 0\}$ is the negative part of the function~$u$, and the constants~$c$ depend only on $n$, $p$, $s$ and on the structural constants~$\lambda$ and $\Lambda$ defined in~\eqref{hp_k}.
\end{thm}
It is worth remarking that in the case when $u$ is nonnegative in the whole~$\R^n$, the inequality in~\eqref{cloe} reduces to the classical Harnack inequality.
\vspace{1.5mm}

We also consider the situation when the function $u$ is merely a weak supersolution to problem~\eqref{problema}; see Definition~\ref{def_weak} below. In analogy to the local case $s=1$, we prove a weak Harnack inequality.
\begin{thm}[\bf Nonlocal weak Harnack inequality]\label{thm_weak}
For any $s\in (0,1)$ and any $p \in (1,\infty)$, let $u\in W^{s,p}(\mathds{R}^n)$ be a weak supersolution to~\eqref{problema} such that $u\geq 0$ in $B_R\equiv B_R({x_0})\subset\Omega$. Then the following estimate holds for any $B_{r}\equiv B_{r}({x_0})\subset B_{R/2}(x_0)$ and for any $t< (p-1)n/(n-sp)$ with $1<p<n/s$,
\begin{equation}\label{eq1}
\left( \dashint_{B_r} u^t \right)^{\!\frac{1}{t}} \, \leq \,
c \inf_{B_{2r}} u 
+ c \left(\frac{r}{R}\right)^{\frac{sp}{p-1}} \text{\rm Tail}(u_-; {x_0},R),
\end{equation}
where ${\rm Tail}(\cdot)$ is defined in \eqref{tail_int}, $u_-=\max\{-u, 0\}$ is the negative part of the function~$u$, and the constants $c$ depend only on $n$, $p$, $s$, $\lambda$ and $\Lambda$.
\end{thm}
As expected,  the contribution given by the nonlocal tail have again to be considered and the result is analogous to the local case if $u$ is nonnegative in the whole~$\R^n$.

\vspace{1.5mm}

For what concerns the main topic in the present paper, i.~\!e., Harnack-type inequalities for the minimizers of~\eqref{funzionale},  a few words for {\it the linear case when $p=2$} 
 have to be said since very interesting results arise comparing to the classic local case when~$s=1$. Firstly, the analog standard elliptic Harnack inequality can be easily derived using Poisson kernels by requiring the minimizers~$u$ to be {\it nonnegative in the whole $\R^n$}. This restriction is evidently very strong  and it also precludes to establish many consequences really not needing such positivity of the solutions to~\eqref{problema}. For instance, even the possibility to directly derive the  tightly related H\"older regularity estimates is ruled out, although it is well known that in the local case both Harnack and H\"older statements are equivalent for a large class of problems.
 For this, during the last decades, the validity of the classical Harnack inequality without extra positivity assumptions has been an open problem in a nonlocal setting, and more in general for integro-differential operators of the form in~\eqref{lisa}. An answer has been recently given by Kassmann, who provided a simple counter-example by showing that positivity cannot be dropped nor relaxed even in the most simple case when~$\mathcal{L}$ coincides with the fractional Laplacian~$(-\Delta)^s$; see Theorem~1.2 in~\cite{Kas07b}. The same author proposed a new formulation of the Harnack inequality without requiring the additional positivity on the whole $\R^n$ by adding an extra term, basically a natural {\it tail} contribution on the right hand-side, which takes into account the nonlocality of the fractional Laplacian, for any $s\in (0,1)$; see Theorem~3.1 in~\cite{Kas11}.

\vspace{1.5mm}

Here,  we will deal with a larger class of operators whose kernel $K$ is not necessarily symmetric, with only measurable coefficients, and, above all, satisfying fractional differentiability  {\it for any} $s\in (0,1)$ and $p$-summability   {\it for any} $p>1$. For this, we will have to handle not only the usual nonlocal character of such fractional operators, as for instance in the aforementioned papers~\cite{Kas07b,Kas11}, but also the difficulties given by the corresponding nonlinear behavior. 
As a consequence, we can make use neither of the powerful ``$s$-harmonic extension framework'' provided by  Caffarelli-Silvestre  in~\cite{CS07}, nor of various  tools as, e.~\!g., 
the sharp 3-commutators estimates introduced in~\cite{DLR11} to deduce the regularity of weak fractional harmonic maps,
the strong barriers and density estimates in~\cite{SV13a,SV13b,PSV13}, 
the commutator and energy estimates in~\cite{PP13,PS13},
 and so~on. Indeed, the aforementioned tools seem not to be trivially adaptable to a nonlinear framework; also, increasing difficulties are due to the non-Hilbertian structure of the involved fractional Sobolev spaces~$W^{s,p}$ when $p$ is different than 2. In fact, we develop a nonlocal counterpart for the seminal paper by DiBenedetto-Trudinger~\cite{DT84}.

\vspace{1.5mm}

Finally, a great attention has been focused on the study of problems involving fractional Sobolev spaces and corresponding nonlocal equations, both from a pure mathematical point of view and for concrete applications, since they naturally arise in many different contexts (see for instance~\cite{DPV12} for an elementary introduction to this topic and for a wide list of related references). However, for regularity and related results for the minimizers of this kind of operators when $p\neq2$, the theory seems to be rather incomplete. Nonetheless, some partial results are known.
It is worth citing the higher regularity contributions in the case when $s$ is close to $1$ proven in the interesting papers~\cite{BCF12,IN10},  recently extended in some extents by the authors in~\cite{DKP13} for any $s\in (0,1)$; see, also,~\cite{CLM12} for related existence and uniqueness results in the case when $p$ goes to infinity.
 Also, we would like to mention the analysis in the papers~\cite{BLP13,FP14,IS14} where some basic results for fractional $p$-eigenvalues have been proven.
 
\vspace{2.5mm}

The paper is organized as follows. In Section~\ref{sec_preliminaries} below, we fix the notation by also recalling some recent results on the fractional $p$-minimizers and some classical tools. Section~\ref{sec_positivity} is devoted to a nonlocal expansion of positivity in order to accurately estimate the infimum of the superminima of~\eqref{funzionale}. In Section~\ref{sec_harnack}, we are finally able to complete the proof of Theorem~\ref{harnack}.
In Section~\ref{sec_weak}, we shall prove the weak Harnack inequality given by Theorem~\ref{thm_weak}.

\vspace{2mm}
\section{Preliminaries}\label{sec_preliminaries}

In this section we state the general assumptions on the quantity we are dealing with. We keep these assumptions throughout the paper.

\vspace{1mm}

The {\it kernel} $K:\R^n\times \R^n \to [0,\infty)$ is a measurable function such that
\begin{equation}\label{hp_k}
\lambda \leq K(x,y)|x-y|^{n+sp} \leq \Lambda \ \text{for almost} \ x, y \in \R^n,
\end{equation}
for some $s\in (0,1)$, $p>1$, $\lambda\geq\Lambda \geq1$. 
We notice that the assumption on $K$ can be weakened as follows
\begin{equation}\label{hp_k1}
\lambda \leq K(x,y)|x-y|^{n+sp} \leq \Lambda \ \text{for almost} \ x, y \in \R^n \ \text{s.~\!t.} \ |x-y| \leq 1,
\end{equation}
\begin{equation}\label{hp_k2}
0\leq K(x,y)|x-y|^{n+\eta} \leq M \ \text{for almost} \ x, y \in \R^n \ \text{s.~\!t.} \ |x-y| > 1,
\end{equation}
for some $s, \lambda, \Lambda$ as above, $\eta>0$ and $M\geq 1$, as seen, e.~\!g., in the recent papers by Kassmann (see also the more general assumptions in 
\cite{Kas11}). For the sake of simplicity, we will work under the assumption in~\eqref{hp_k}; the assumptions in~\eqref{hp_k1}-\eqref{hp_k2} would bring no relevant differences in all the proofs in the rest of the paper.
 
Now we recall the definition of the fractional Sobolev spaces, denoted by $W^{s,p}(\mathds{R}^n)$. For any $p\in[1,\infty)$ and $s\in(0,1)$
$$
W^{s,p}(\mathds{R}^n):=\left\{ v\in L^{p}(\mathds{R}^n)\,:\, \frac{|v(x)-v(y)|}{|x-y|^{\frac np+s}}\in L^p(\mathds{R}^n\times \mathds{R}^n)\right\};
$$ 
i.~\!e., an intermediary Banach space between $L^p(\mathds{R}^n)$ and $W^{1,p}(\mathds{R}^n)$ endowed with the natural norm
$$
\|v\|_{W^{s,p}(\mathds{R}^n)}:=\left(\int_{\mathds{R}^n} |v|^p\,{\rm d}x+\int_{\mathds{R}^n}\int_{\mathds{R}^n}\frac{|v(x)-v(y)|^p}{|x-y|^{n+sp}}\,{\rm d}x{\rm d}y\right)^{\frac1 p}.
$$
In a similar way, it is possible to define the fractional Sobolev spaces $W^{s,p}(\Omega)$ in a domain $\Omega \subset\mathds{R}^n$. 
For the basic properties of these spaces and some related topics we refer to \cite{DPV12} and the references therein. 

\vspace{2mm}

For any $u,v \in W^{s,p}(\R^n)$, we consider the functional $\mathcal{E}$ defined by
\begin{equation*}\label{def_energia}
\mathcal{E}(u,v):= \int_{\R^n}\!\int_{\R^n} K(x,y)|u(x)-u(y)|^{p-2}(u(x)-u(y)) (v(x)-v(y))\dxy.
\end{equation*}

Suppose that $u$ and $\varphi$ are sufficiently smooth, take e.~\!g. $C_0^\infty(\R^n)$, and define the linear operator $\mathcal{L}$ as the one satisfying the relation
$$
\langle \mathcal{L}u, \varphi\rangle \, = \, \mathcal{E}(u,\varphi).
$$
Thus, assuming that $K$ satisfies \eqref{hp_k}, for any $u \in W^{s,p}(\R^n)$, we have
$$
\mathcal{L}u(x) = P.~\!V. \int_{\R^n}K_{\text{\rm sym}}(x,y)|u(x)-u(y)|^{p-2}(u(x)-u(y))\,{\rm d}y, \quad x\in \R^n,
$$
up to a multiplicative constant; $P.~\!V.$ being a commonly used abbreviation for ``in the principal value sense''. 

\vspace{1mm}

Let $\Om\subset \R^n$ be a bounded open set in $\mathds{R}^n$. Let $g\in W^{s,p}(\mathds{R}^n)$, we are interested in weak solutions to the following class of integro-differential equations
\begin{equation}\label{pb}
\dys
\begin{cases}
\mathcal{L} u = 0 & \text{in} \ \Omega, \\[1ex] 
u = g  & \text{in} \ \R^n\setminus\Omega.
\end{cases}
\end{equation}

As customary, a function $u\in W^{s,p}(\mathds{R}^n)$ is a solution to~\eqref{pb} if  
$\mathcal{E}(u,\varphi)=0$ for all test function $\varphi \in C^\infty_0(\Om)$. Moreover, if we consider the following functional
\begin{equation}\label{def_F}
\mathcal{F}(v)=\int_{\mathds{R}^n}\int_{\mathds{R}^n} K(x,y)|v(x)-v(y)|^p\,{\rm d}x{\rm d}y,
\end{equation}
thanks to the assumptions \eqref{hp_k} on the kernel $K$, there exists a unique $p$-minimizer $u$ of $\mathcal{F}$ over all $v \in W^{s,p}(\mathds{R}^n)$ such that $v=g$ in $\mathds{R}^n\setminus \Omega$ and it is a weak solutions to problem \eqref{pb} and vice versa; see Theorem~2.3 in \cite{DKP13}.

We conclude this section by recalling the definitions of weak subsolution and weak supersolution to problem~\eqref{pb}. Before, we define, for a given $g\in W^{s,p}(\mathds{R}^n)$, the convex sets of $W^{s,p}(\mathds{R}^n)$ as
$$
\mathcal{K}_{g}^{\pm}(\Omega):=\{v\in W^{s,p}(\mathds{R}^n)\,:\, (g-v)_{\pm}\in W^{s,p}_0(\Omega)\}
$$ 
and
$$
\mathcal{K}_{g}(\Omega):=\mathcal{K}_{g}^{+}(\Omega)\cap \mathcal{K}_{g}^{-}(\Omega)=\{v\in W^{s,p}(\mathds{R}^n)\,:\, v-g\in W^{s,p}_0(\Omega)\},
$$ 
where we denoted by $W^{s,p}_0(\Omega)$  the closure of $C^\infty_0(\Omega)$ in the norm $\|\cdot\|_{W^{s,p}(\Omega)}$. We underline that the functions in the space $W^{s,p}_0(\Omega)$ are defined in the whole space, since they are considered to be extended to zero outside $\Omega$.
\begin{defn}\label{def_weak}
Let $g\in W^{s,p}(\mathds{R}^n)$. A function $u\in \mathcal{K}^-_g$ is a \textit{weak subsolution} to problem \eqref{pb} if 
\begin{equation*}\label{weak_sub}
\int_{\mathds{R}^n}\int_{\mathds{R}^n} K(x,y)|u(x)-u(y)|^{p-2}(u(x)-u(y))(\eta(x)-\eta(y))\,{\rm d}x{\rm d}y\leq 0
\end{equation*}
for every nonnegative $\eta\in W^{s,p}_0(\mathds{R}^n)$.
\\ A function~$u\in \mathcal{K}^+_g$ is a \textit{weak supersolution} to problem \eqref{pb} if 
\begin{equation}\label{weak_super}
\int_{\mathds{R}^n}\int_{\mathds{R}^n} K(x,y)|u(x)-u(y)|^{p-2}(u(x)-u(y))(\eta(x)-\eta(y))\,{\rm d}x{\rm d}y\geq 0
\end{equation}
for every nonnegative $\eta\in W^{s,p}_0(\mathds{R}^n)$.
As customary, a function~$u\in \mathcal{K}_g$ is a \textit{weak solution} to problem \eqref{pb} if it is both a sub and a supersolution; that is,
\begin{equation*}\label{weak_sol}
\int_{\mathds{R}^n}\int_{\mathds{R}^n} K(x,y)|u(x)-u(y)|^{p-2}(u(x)-u(y))(\eta(x)-\eta(y))\,{\rm d}x{\rm d}y=0
\end{equation*} 
for every $\eta\in W^{s,p}_0(\mathds{R}^n)$.
\end{defn}
 \vspace{2mm}
 
Similarly, it is possible to define sub- and superminimizers of \eqref{def_F}, see for instance Definition~2.2 in \cite{DKP13}.

\subsection{Notation}\label{sec_notation}

Before starting with the proofs, it is convenient to fix some notation which will be used throughout the rest of the paper. Firstly, notice that we will follow the usual convention of denoting by $c$ a general positive constant which will not necessarily be the same at different occurrences and which can also change from line to line. For the sake of readability, dependencies of the constants will  be often omitted within the chains of estimates, therefore stated after the estimate. 
Relevant dependences on parameters will be emphasized by using parentheses; special constants will be denoted by $c_0$, $c_1$,...

As customary, we denote by
$$
B_R({x_0})=B({x_0}; R):=\{x\in\mathds{R}^n : |x-{x_0}|<R\}
$$
the open ball centered in ${x_0}\in\mathds{R}^n$ with radius $R>0$. When not important and clear from the context, we shall use the shorter notation $B_R:=B({x_0};R)$. 
Moreover, if $f\in L^1 (S)$ and the $n$-dimensional Lebesgue measure $|S|$ of the set $S\subseteq \mathds{R}^n$ is finite and strictly positive, we write
\begin{equation*}\label{mean}
(f)_S:=\dashint_{S} f(x)\,{\rm d}x=\frac{1}{|S|}\int_S f(x)\,{\rm d}x.
\end{equation*}

Let $k>0$ and $S\subseteq \mathds{R}^n$, we denote by
\begin{equation}\label{w+}
w_+(x):=(u(x)-k)_+=\max\big\{ u(x)-k,0\big\}, 
\end{equation}
and
\begin{equation}\label{w-}
w_-(x):=(u(x)-k)_-=(k-u(x))_+,  
\end{equation}
for any $x\in S$.
Clearly $w_+(x)\neq 0$ in the set $\big\{x\in S:u(x)>k\big\}$, and $w_-(x)\neq 0$ in the set $\big\{x\in S:u(x)<k\big\}$. 

Finally, in order to deal also with non symmetric kernels, we define 
\begin{equation}\label{bark}
\bar{K}(x,y)=\max\big\{K(x,y),K(y,x)\big\}.
\end{equation}

We now recall the definition of the {\it nonlocal tail of a function $u$ in the ball $B_{R}({x_0})$}, as seen in~\cite{DKP13}. As mentioned in the introduction, this quantity will play an important role in the rest of the paper. For any $u\in W^{s,p}(\mathds{R}^n)$ and $B_R({x_0})\subset \mathds{R}^n$ we write
\begin{equation}\label{def_tail}
\text{\rm Tail}(u; {x_0}, R) := \left[R^{sp}\left( \int_{\mathds{R}^n \setminus B_R} |u(y)|^{p-1}|y-{x_0}|^{-n-sp}\, {\rm d}y\right)\right]^{\frac{1}{p-1}}.
\end{equation}

\vspace{3mm}
\subsection{Some recent results on the fractional $p$-minimizers}\label{sec_known}

In this section, we recall some recent results for the minimizers of the nonlocal functionals~\eqref{def_F} and hence also for the weak solutions to problem \eqref{pb}, which can be found in~\cite{DKP13}.

\vspace{2mm}

Firstly, we state a general inequality proved in~\cite{DKP13}, which shows that the natural extension of the Caccioppoli inequality to the nonlocal framework has to take into account a suitable {\it tail}. For other fractional Caccioppoli-type inequalities, see also~\cite{Min07,Min11} and \cite{FP14}.
\begin{thm}{\rm (\cite[Theorem~1.4]{DKP13}).}\label{lem_caccio}
Let $p \in [1,\infty)$ and
let $u\in W^{s,p}(\mathds{R}^n)$ be a weak solution to problem \eqref{pb}.
 Then, for any $B_r\equiv B_r({x_0})\subset \Omega$ and any nonnegative~$\varphi\in C^\infty_0(B_r)$,  the following estimate holds true
\begin{eqnarray}\label{cacio1}
\nonumber && \int_{B_r}\int_{B_r}K(x,y) |w_{\pm}(x)\varphi(x)-w_{\pm}(y)\varphi(y)|^p \,{\rm d}x{\rm d}y\\
 &&\qquad \qquad \quad \leq c\int_{B_r}\int_{B_r} \bar{K}(x,y)
 (\max\{w_{\pm}(x),w_\pm(y)\})^p |\varphi(x)-\varphi(y)|^p\,{\rm d}x{\rm d}y\\
&&\qquad \qquad \quad \quad+\,c \,\int_{B_r}w_{\pm}(x)\varphi^p(x)\,{\rm d}x \left(\sup_{y\,\in\, {\rm supp}\,\varphi}\int_{\mathds{R}^n\setminus B_r}\bar{K}(x,y)w_{\pm}^{p-1}(x)\,{\rm d}x \right)\!, 
\nonumber
\end{eqnarray}
where $w_{\pm}$, $\bar{K}$ are defined in \eqref{w+}-\eqref{w-} and \eqref{bark}  respectively, and $c$ depends only on~$p$.
\end{thm}
\begin{rem}\label{rem_caccioppoli}
We underline that the estimate in~\eqref{cacio1} holds for $w_+$ also when $u$ is  a weak subsolution to \eqref{pb} and for $w_-$ when $u$ is a weak supersolution to \eqref{pb}.
\end{rem}

As in the local case, the estimate above contains basically all the information deriving from the minimum property of the functions $u$ for what concerns the corresponding H\"older continuity.
A first natural consequence is the local boundedness of both  $p$-subminimizers of~\eqref{def_F} and weak subsolutions to problem \eqref{pb}, as stated below.
\begin{thm}{\rm (\cite[Theorem~1.1 and Remark 4.2]{DKP13}).}\label{thm_local}
Let $p \in [1,\infty)$, let $u\in W^{s,p}(\mathds{R}^n)$ be a weak subsolution to problem~\eqref{pb} and let $B_r\equiv B_r({x_0})  \subset \Omega$. Then the following estimate holds true
\begin{eqnarray}\label{sup_estimate}
\sup_{B_{r/2}}u \, \leq \, c\,\delta\, {\rm Tail}(u_+;{x_0},r/2)+c\, \delta^{-\frac{(p-1)n}{s p^2}} \left(\dashint_{B_r} u_+^p\,{\rm d}x\right)^{\frac 1p},
\end{eqnarray}
where ${\rm Tail}(\cdot)$ is defined in \eqref{def_tail}, $u_+=\max\{u,0\}$ is the positive part of the function $u$, the parameter~$\delta \in (0,1]$, and the constants $c$ depend only on $n$, $p$, $s$, $\lambda$ and~$\Lambda$.
\end{thm} 
Combining Theorem~\ref{lem_caccio} together with a nonlocal {\it Logarithmic-Lemma} (see \cite[Lem-~\break ma~1.3]{DKP13}), one can prove that the both $p$-minimizers and weak solutions enjoy an oscillation estimates, which naturally yields the desired H\"older continuity. Once again, the tail contribution given by the nonlocal form of the involved operators has to be taken into account (see~\cite[Theorem 1.2]{DKP13}).

\vspace{2mm}

\subsection{Classical technical tools}\label{sec_classic}

In this section we collect some classical tools that will be useful in the proofs of the main results of the paper.

Below, a Krylov-Safonov covering lemma, whose proof can be found, for instance, in~\cite[Lemma 7.2]{KS01}. 
We have 
\begin{lemma}\label{lem_KS}
Let $E\subset B_r({x_0})$ a measurable set. Let $\bar{\delta}\in (0,1)$, and define
\begin{equation}\label{E_delta}
[E]_{\bar{\delta}}:= \bigcup_{\rho>0} \left\{B_{3\rho}(x)\cap B_r({x_0}),\,\,\, x\in  B_r({x_0}) \,:\, |E\cap B_{3\rho}(x) |> \delta |B_\rho(x)|\right\}.
\end{equation}
Then, either
$$
\textrm{ i\,{\rm)}} \ \left|[E]_{\bar{\delta}}\right|\geq \frac{c_3}{\bar{\delta}} |E|
$$
or  
$$
\hspace{-0.5cm}\textrm{ ii\,{\rm)}}\ [E]_{\bar{\delta}}= B_{r}({x_0}),
$$
where $c_3=c_3(n)$.
\end{lemma}

Two well-known iteration lemmata are also needed.
\begin{lemma}\label{lem_iter} {\rm (see, e.~\!g., \cite[Lemma 7.1]{Giu03}).}
Let $\beta>0$ and let $\{ A_j\}$ be a sequence of real positive numbers such that
$$
A_{j+1} \, \leq \, c_0\, b^j A_j^{1+\beta}
$$
with $c_0>0$ and $b>1$. 
\noindent
\\ If $\dys A_0 \, \leq \, c_0^{-\frac{1}{\beta}} \,b^{-\frac{1}{\beta^2}}$, then we have
$$
A_j \, \leq \, b^{-\frac{j}{\beta}}\,A_0 ,
$$
which in particular yields $\dys \lim_{j\to \infty} A_j = 0$.
\end{lemma}

\begin{lemma} {\rm (see, e.~\!g., \cite[Lemma 1.1]{GG82}).}\label{lem_gg}
Let $f=f(t)$ be a nonnegative bounded function defined for $0 \leq T_0 \leq t \leq T_1$. Suppose that for $T_0 \leq t < \tau \leq T_1$ we have
$$
f(t) \, \leq \, c_1(\tau-t)^{-\theta} + c_2 + \zeta f(\tau),
$$
where $c_1$, $c_2$, $\theta$ and $\zeta$ are nonnegative constants, and $\zeta < 1$. Then there exists a constant $c$, depending only on $\theta$ and $\zeta$, such that for every $\rho$, $R$, $T_0 \leq \rho < R \leq T_1$, we have
$$
f(\rho) \, \leq \, c\left[c_1\,(R-\rho)^{-\theta}+ c_2\right].
$$
\end{lemma}

\vspace{2mm}

\section{Towards a Harnack inequality: expansion of positivity}\label{sec_positivity} 

In this section, we show that we can accurately estimate the infimum of the super-minima of~\eqref{funzionale} and of the weak supersolutions to problem~\eqref{problema}. Our strategy extends the analogous expansion of positivity in the local framework~$s=1$, as presented, e.~\!g., in~\cite[Section 7.5]{Giu03}.  
Clearly, in order to extend the results there to our framework, we have to take into account considerable and decisive modifications to handle the nonlocality of our problems.

From now on, for the sake of readability, we define
\begin{equation*} 
{\rm d}\nu:=K(x,y)\,{\rm d}x{\rm d}y\quad \text{and} \quad {\rm d}\bar{\nu}:= \bar{K}(x,y)\,{\rm d}x{\rm d}y,\,\,\, \text{with }\bar{K}\text{ as in }\eqref{bark}.
\end{equation*}
\begin{lemma}\label{thm_exp}
Let $u\in W^{s,p}(\mathds{R}^n)$ be a weak supersolution to problem~\eqref{pb} such that $u\geq 0$ in $B_R({x_0})\subset \Om$. Let $k \geq 0$. Suppose that there exists $\sigma \in (0,1]$ such that
\begin{equation}\label{hp0}
\dys
| B_r \cap \{ u \geq k \} | \, \geq \, \sigma |B_r|,
\end{equation}
for some $r$ satisfying $0<16r<R$. Then there exists a constant~$\bar c \equiv \bar c(n,s,p,\lambda,\Lambda)$ such that 
\begin{equation*}\label{eq_utilde}
\dys
\left| B_{6r} \cap \left\{ u \leq 2\delta k -  \frac{1}{2} \left( \frac{r}{R} \right)^{\frac{sp}{p-1}} \text{\rm Tail}(u_-; {x_0}, R) \right\}\right| 
\, \leq \, \frac{\bar c}{\sigma \log{\frac{1}{2\delta}}}\, |B_{6r}|
\end{equation*}
holds for all $\delta \in (0,1/4)$, where ${\rm Tail}(\cdot)$ is defined in \eqref{def_tail}.
\end{lemma}
\begin{proof}
To begin, set
\[
d := \frac{1}{2} \left( \frac{r}{R} \right)^{\frac{sp}{p-1}} \text{\rm Tail}(u_-; {x_0}, R) \qquad \mbox{and} \qquad
\tilde u = u + d,
\]
so that $\tilde u$ is obviously still a supersolution. 
Take a smooth function~$\varphi$ with support in $B_{7r}$ such that $0\leq \varphi \leq 1$ in $B_{7r}$, $\varphi\equiv 1$ in $B_{6r}$ and $|D\varphi| \leq c/r$. By choosing $\eta:= \tilde{u}^{1-p} \varphi^p$ in~\eqref{weak_super}, we get
\begin{eqnarray}\label{eq_I1I2I3}
\dys
0 & \leq & \int_{B_{8r}}\int_{B_{8r}}|\tilde{u}(x)-\tilde{u}(y)|^{p-2}(\tilde{u}(x)-\tilde{u}(y))
(\tilde{u}^{1-p}(x)\varphi^p(x) - \tilde{u}^{1-p}(y)\varphi^p(y))\dn \nonumber \\[1ex] 
& &  + \, \int_{\mathds{R}^n \setminus B_{8r}}\int_{B_{8r}}|\tilde{u}(x)-\tilde{u}(y)|^{p-2}(\tilde{u}(x)-\tilde{u}(y))
\tilde{u}^{1-p}(x)\varphi^p(x)\dn \nonumber \\[1ex]
&& - \, \int_{B_{8r}}\int_{\mathds{R}^n \setminus B_{8r}}|\tilde{u}(x)-\tilde{u}(y)|^{p-2}(\tilde{u}(x)-\tilde{u}(y))
\tilde{u}^{1-p}(y)\varphi^p(y)\dn \nonumber \\[1.5ex]
& =: & I_1 + I_2 +I_3.
\end{eqnarray}

The first integral  
can be estimated like $I_1$ in the proof of \cite[Lemma 1.3]{DKP13} (more precisely, see (3.12) and (3.17) there), in order to get
\begin{equation*}\label{eq_I1}
\dys
I_1 \, \leq \, - \frac1c \int_{B_{6r}}\int_{B_{6r}} \left| \log\left(\frac{\tilde{u}(x)}{\tilde{u}(y)}\right)\right|^p \dn + c\,r^{n-sp}.
\end{equation*}
\vspace{1mm}

It remains to estimate the second integral in the right hand-side of~\eqref{eq_I1I2I3}, which in turn will imply an estimate for $I_3$, too.
Firstly, we split $I_2$ as follows
\begin{eqnarray*} 
\nonumber
I_{2} & = & \int_{\mathds{R}^n \setminus B_{8r} \cap \{\tilde{u}(y) < 0\}}\int_{B_{8r}}  
|\tilde{u}(x)-\tilde{u}(y)|^{p-2}(\tilde{u}(x)-\tilde{u}(y))
\tilde{u}^{1-p}(x)\varphi^p(x)\dn
\\ 
\nonumber 
& & + \int_{\mathds{R}^n \setminus B_{8r} \cap \{\tilde{u}(y) \geq 0\}}\int_{B_{8r}}  
|\tilde{u}(x)-\tilde{u}(y)|^{p-2}(\tilde{u}(x)-\tilde{u}(y))
\tilde{u}^{1-p}(x)\varphi^p(x)\dn
\\[1ex]
 \nonumber & =: & I_{2,1} +  I_{2,2} .
\end{eqnarray*}
By the definition of $\tilde{u}$ and so of $d$,  the assumption on the kernel and using the fact that $\varphi$ is supported in $B_{7r}$, we get
\begin{eqnarray*} 
\nonumber I_{2,1} & = & 
\int_{\mathds{R}^n \setminus B_{8r} }\int_{B_{8r}}   (\tilde{u}(x) + (\tilde{u}(y))_-)^{p-1} \tilde{u}^{1-p}(x)\varphi^p(x)\dn
\\[1ex]
 \nonumber & \leq & c r^n  \int_{\mathds{R}^n \setminus B_{8r} } \left(1 + \frac{(u(y))_-}{d} \right)^{p-1} |y-{x_0}|^{-n-sp}  \, dy
\\[1ex]
 \nonumber & \leq & c r^n r^{-sp} + c r^n d^{1-p}  R^{-sp} [\text{\rm Tail}(u_-; {x_0}, R)]^{p-1}
\\[1ex]
 \nonumber & \leq &  cr^n r^{-sp} .
\end{eqnarray*}
On the other hand, since $u(y)$ is nonnegative whenever $y \in B_{7r}$, 
we easily deduce  that
\[
I_{2,2} \leq c r^{n-sp}\,.
\]
Therefore, we actually obtain
\[
I_{2} + I_3 \leq c r^{n-sp}
\]
for a constant $c \equiv c(n,s,p,\lambda,\Lambda)$.
Merging the  estimates
above, we conclude with the following intermediate estimate
\begin{equation}\label{eq999}
\int_{B_{6r}}\int_{B_{6r}} \left| \log\left(\frac{\tilde{u}(x)}{\tilde{u}(y)}\right)\right|^p\dn  \leq c r^{n-sp}\,.
\end{equation}
Now, for any $\delta\in (0,\,1/4)$, define
$$
\dys
v:= \left[\min \bigg\{ \log{\frac{1}{2\delta}},\, \log \frac{k + d}{\tilde{u}} \bigg\}\right]_{+}\! .
$$
Since $v$ is a truncation of $\log(k+d) - \log \tilde u $, the energy decreases and, in particular, 
$$
\int_{B_{6r}}\int_{B_{6r}} \left|v(x) - v(y)\right|^p \dn
\,  \leq\, 
\int_{B_{6r}}\int_{B_{6r}} \left| \log\left(\frac{\tilde{u}(x)}{\tilde{u}(y)}\right)\right|^p\dn  
\,\leq \,c r^{n-sp}
$$
holds, in view of~\eqref{eq999}.
 Then by H\"older's inequality and fractional Poincar\'e inequality  (see, e.~g., Section~4 in~\cite{Min03}) we also deduce that
\begin{equation}\label{eq_v}
\int_{B_{6r}}|v(x)-(v)_{B_{6r}}|\,{\rm d}x \leq cr^{s+n/p'} \left[\int_{B_{6r}}\int_{B_{6r}} |v(x)-v(y)|^p \,{\rm d}\nu \right]^{1/p}\leq  c\, |B_{6r}|.
\end{equation}
Notice that by the definitions of $v$ and $\tilde{u}$ we have
$$
\{ v=0 \} =  \{\tilde{u} \geq k+d \}=\{u\geq k \} .
$$
Hence, by assumption~\eqref{hp0}, it follows that
$$
\dys |B_{6r} \cap \{ v= 0 \} |  \geq \frac{\sigma}{6^n} |B_{6r}|.
$$
Following~\cite{MZ97} (see also the proof of Lemma 5.1 in \cite{DKP13}), together with the estimate above, we get
\begin{eqnarray*}
\dys
\log{\frac{1}{2\delta}} & = & \frac{1}{|B_{6r}\cap \{ v=0 \} |} \int_{B_{6r}\cap \{ v=0 \}} \left( \log{\frac{1}{2\delta}}-v(x)\right) \, {\rm d}x \\[1ex]
& \leq & \frac{6^n}{\sigma}\left[ \log{\frac{1}{2\delta}} -(v)_{B_{6r}}\right].
\end{eqnarray*}
Thus, integrating the previous inequality over $B_{6r} \cap\big \{ v=\log (1/2\delta)\big\}$, we get
$$
\left|\left\{ v=\log{\frac{1}{2\delta}}\right\} \cap B_{6r} \right|\,\log \frac{1}{2\delta} 
\, \leq \, \frac{6^n}{\sigma} \int_{B_{6r}} | v(x)- (v)_{B_{6r}}| \, {\rm d}x 
\, \leq \, \frac{c}{\sigma} \,|B_{6r}|,
$$
where we also used~\eqref{eq_v}.  On the whole, we have proved for all $0<\delta<1/4$ that
\begin{equation*}\label{eq_utilde2}
\dys
|  B_{6r} \cap \{ \tilde{u}\leq 2\delta(k+d) \}| 
\, \leq \, \frac{c}{\sigma} \, \frac{1}{\log{\frac{1}{2\delta}}}\, | B_{6r}|;
\end{equation*}
thus inserting the definition of $\tilde u$ into the display above
finishes the proof.
\end{proof}
\vspace{1.5mm}

The main result of this section is condensed in the following
\begin{lemma}\label{thm_exp2}
Let $u\in W^{s,p}(\mathds{R}^n)$ be a weak supersolution to problem~\eqref{pb} such that $u\geq 0$ in $B_R({x_0})\subset \Om$. Let $k\geq 0$ and suppose that there exists $\sigma\in (0,1]$  such that
\begin{equation*}\label{hp0_bis}
\dys
| B_r \cap \{ u \geq k \} | \, \geq \, \sigma |B_r|,
\end{equation*}
for some $r$ satisfying $0<16 r<R$.  Then there exists a constant~$\delta\in (0,1/4)$ depending only on $n$, $p$, $s$, $\lambda$, $\Lambda$, $\sigma$, for which
\begin{equation}\label{eq_exp}
\dys
\inf_{B_{4r}} u \, \geq \, \delta k -   \left(\frac{r}{R}\right)^{\frac{sp}{p-1}} \text{\rm Tail}(u_-; {x_0}, R) 
\end{equation}
holds. Here ${\rm Tail}(\cdot)$ is defined in \eqref{def_tail}.
\end{lemma}
\begin{proof}
Without loss of generality, we may assume that 
\begin{equation}\label{tail cond exp pos}
 \left(\frac{r}{R}\right)^{\frac{sp}{p-1}} \text{\rm Tail}(u_-; {x_0}, R) \leq \delta k,
\end{equation}
since otherwise~\eqref{eq_exp} trivializes by the nonnegativity of $u$ in $B_R$. 

Now, for any $r\leq \rho \leq 6r$, take a smooth function $\varphi$ with support in $B_\rho$ and consider the  test function
$
\eta: = w_- \varphi^p,
$
where we have denoted by $w_-:=(\ell-u)_{+}$, for any $\ell \in (\delta k , 2\delta k)$. By testing~\eqref{weak_super}, we get
\begin{eqnarray*}\label{eq_10star}
\dys
0 & \leq & \int_{B_\rho} \int_{B_\rho} |u(x)-u(y)|^{p-2}( u(x)-u(y))
(w_-(x) \varphi^p(x) - w_-(y)\varphi^p(y)) \dn \nonumber \\[1ex]
&& +\int_{\mathds{R}^n \setminus B_\rho} \int_{B_\rho} |u(x)-u(y)|^{p-2}( u(x)-u(y))
w_-(x) \varphi^p(x) \dn \nonumber \\[1ex]
&& -\int_{B_\rho} \int_{\mathds{R}^n \setminus B_\rho} |u(x)-u(y)|^{p-2}( u(x)-u(y))
w_-(y) \varphi^p(y) \dn \nonumber \\[1.5ex]
& =: & J_1 + J_2 + J_3.
\end{eqnarray*}

As before, it is convenient to split the second (and analogously the third) integral in the right hand-side of the preceding inequality as
\begin{eqnarray*}
J_2 & = & 
\int_{\mathds{R}^n \setminus B_\rho \cap \{u(y) <  0\}} \int_{B_\rho} | u(x)-u(y)|^{p-2}(u(x)-u(y)) w_-(x) \varphi^p(x) \dn \\[1ex]
&& + \int_{\mathds{R}^n \setminus B_\rho \cap \{u(y) \geq  0\} }\int_{B_\rho} |u(x)-u(y)|^{p-2}(u(x)-u(y)) w_-(x) \varphi^p(x) \dn \\[1.5ex]
& =: & J_{2,1} + J_{2,2}.
\end{eqnarray*}

Let us estimate the integral $J_{2,1}$. Notice that 
\begin{eqnarray*}
\dys && K(x,y) |u(x)-u(y)|^{p-2} (u(x)-u(y)) w_-(x)\varphi^p(x) \\[1ex]
&& \qquad\qquad\qquad \qquad \leq \,(\ell + (u(y))_-)^{p-1} \ell \left( \sup_{x\in \text{supp}\,\varphi} \bar{K}(x,y) \right) \chi_{B_\rho \cap \{ u < \ell\}}(x), 
\end{eqnarray*}
where $\bar{K}$ is defined in~\eqref{bark}. This plainly yields
\begin{equation}\label{eq_12star}
\dys
J_{2,1} \, \leq  \,  \ell \left( \sup_{x\in \text{supp}\,\varphi} \int_{\mathds{R}^n \setminus B_\rho} (\ell + (u(y))_-)^{p-1} \bar{K}(x,y) \,{\textrm d}y \right) |B_\rho \cap \{ u < \ell\}|. 
\nonumber
\end{equation}
For the contribution given by $J_{2,2}$ we instead have, using the nonnegativity of $u$ in $B_\rho$,
\begin{eqnarray*}
\dys && K(x,y) |u(x)-u(y)|^{p-2} (u(x)-u(y)) w_-(x)\varphi^p(x) \\[1ex]
&& \qquad\qquad\qquad\qquad \qquad\qquad\qquad\leq \, \ell^p \left(\sup_{x\in {\rm supp}\,\varphi} \bar{K}(x,y)\right) \chi_{B_\rho \cap \{ u < \ell\}}(x).
\end{eqnarray*}
As the similar reasoning it holds for $J_3$ as well, we deduce that
$$
J_2 + J_3 
\, \leq\, c\, \ell \left( \sup_{x\in \text{supp}\,\varphi} \int_{\mathds{R}^n \setminus B_\rho} (\ell + (u(y))_-)^{p-1} \bar{K}(x,y) \,{\textrm d}y \right) |B_\rho \cap \{ u < \ell\}|. 
$$
 The integral in $J_1$ can be instead estimated as follows (as one can check in the proof of the Caccioppoli-type estimate in \cite[Theorem~1.4]{DKP13}).
\begin{eqnarray*}
\dys
J_1 & \leq & 
-c \int_{B_\rho} \int_{B_{\rho}} |w_-(x) \varphi(x) - w_-(y) \varphi(y) |^p \dn \\[1ex]
&& + c \int_{B_\rho} \int_{B_{\rho}} \big(\max\big\{w_-(x), w_-(y)\big\}\big)^p\, |\varphi(x) -\varphi(y) |^p \, {\textrm d} \bar{\nu}.
\end{eqnarray*}
By combining all the estimates above we finally arrive at
\begin{eqnarray}\label{eq_16star}
\nonumber
&&  \int_{B_\rho} \int_{B_{\rho}} |w_-(x) \varphi(x) - w_-(y) \varphi(y) |^p \dn 
\\ \nonumber   
& & \qquad\qquad\qquad   \leq   \, c \int_{B_\rho} \int_{B_{\rho}} \big(\max\big\{w_-(x),\,w_-(y)\big\}\big)^p\, |\varphi(x) -\varphi(y) |^p \, {\textrm d} \bar{\nu}
\\ & & \qquad\qquad\qquad \quad + \, c\, \ell \left( \sup_{x\in \text{supp}\,\varphi} \int_{\mathds{R}^n \setminus B_\rho} (\ell + (u(y))_-)^{p-1} \bar{K}(x,y) \,{\textrm d}y \right) |B_\rho \cap \{ u < \ell\}|\,.
\end{eqnarray}

At this level, we need to set the quantities in~\eqref{eq_16star} in order to apply Lemma~\ref{lem_iter}. To this end, let 
$$
\ell \equiv \ell_j := \delta k  + 2^{-j-1}\delta k,
$$
and
$$
\dys
\rho \equiv \rho_j:= 4r+2^{1-j} r   \qquad \text{and} \qquad  \widetilde \rho_j= \frac{\rho_{j+1}+\rho_j}{2}
$$
for all $j = 0,1,\ldots$ Note that $\rho_j, \widetilde \rho_j \in (4r,6r)$ and
\[
\ell_{j} - \ell_{j+1} 
\,=\, 2^{-j-2} \delta k 
\, \geq \, 2^{-j-3} \ell_j 
\]
for all such $j$.  Moreover, by~\eqref{tail cond exp pos} we see that 
$$
\ell_0 = \frac{3}{2} \delta k 
\, \leq \, 2\delta k - \frac{1}{2} \left(\frac{r}{R}\right)^{\frac{sp}{p-1}} \text{\rm Tail}(u_-; {x_0}, R)
$$
and hence
\[
\{u<\ell_0\} \subset \left\{u< 2\delta k - \frac{1}{2} \left(\frac{r}{R}\right)^{\frac{sp}{p-1}} \text{\rm Tail}(u_-; {x_0}, R)\right\}\,.
\]
Lemma~\ref{thm_exp} then implies that 
\begin{equation} \label{eq:exp pos iter init}
\frac{|B_{6r} \cap \{u<\ell_0\}| }{|B_{6r}|} \leq \frac{\bar c}{\sigma} \frac1{\log \frac1{2\delta}}. 
\end{equation}
Furthermore, we have for any $j=0,1,2,...$ that
$$
w_- \equiv w_j  =  (\ell_j - u)_{+}
\, \geq \, (\ell_j - \ell_{j+1}) \chi_{\{u < \ell_{j+1}\}} 
\,\geq\, 2^{-j-3} \ell_j \chi_{\{u < \ell_{j+1}\}} .
$$
Let us denote by $B_j:=B_{\rho_j}({x_0})$ and let $\varphi_j \in C^\infty_0(B_{\widetilde \rho_j})$ be such that $0\leq \varphi_j \leq 1$, $\varphi_j\equiv 1$ in $B_{j+1}$, $|D\varphi_j| \leq 2^{j+3}/r$.
With these choices in our hands, we can write
\begin{eqnarray}\label{eq_17star}
\dys
(\ell_j-\ell_{j+1})^p \left( \frac{ |B_{j+1} \cap \{ u < \ell_{j+1}\}|}{|B_{j+1}|} \right)^{\frac{p}{p^\ast}}
&\! \!\leq &\!
\left[ \dashint_{B_{j+1}} w_{j}^{p^\ast} \varphi_j^{p^\ast} \,{\rm d}x \right]^{\frac{p}{p^\ast}}  \leq  c\left[ \dashint_{B_{j}} w_{j}^{p^\ast} \varphi_j^{p^\ast} \,{\rm d}x \right]^{\frac{p}{p^\ast}} \nonumber \\[1ex]
& \!\!\leq & \!c r^{sp} \dashint_{B_j}\int_{B_j} | w_j(x) \varphi_j(x) - w_j(y)\varphi_j(y)|^p \dn,
\end{eqnarray}
where in the last inequality we also used the fractional Sobolev embedding with $p^*=np/(n-sp)$, well defined if $sp<n$.

We then proceed to estimate~\eqref{eq_17star} with the aid of~\eqref{eq_16star}. First, by the properties of $K$ we find the estimate
\begin{eqnarray*}
&&\int_{B_j} \int_{B_{j}} \big(\max\big\{w_j(x),w_j(y)\big\}\big)^p\, |\varphi_j(x) -\varphi_j(y) |^p \, {\textrm d} \bar{\nu}\\
&&\quad \quad\quad\quad\quad\quad \quad\quad \quad  \leq c\,\ell_j^p \int_{B_j}\int_{B_j\cap \{u <\ell_j\}}\|D\varphi_j\|_{\infty}^p|x-y|^{p-n-sp}\,{\rm d}x{\rm d}y\\[1ex]
&&\quad \quad \quad\quad\quad\quad\quad\quad \quad  \leq c\,2^{jp} \ell_j^p \, r^{-sp}|B_j\cap\{ u <\ell_j\}|\,.
\end{eqnarray*}
Second, using for any $y\in \mathds{R}^n \setminus B_j$,
\[
\sup_{x\in \text{supp}\,\varphi_j} \bar{K}(x,y) \leq c 2^{j(n+sp)} |y-x_0|^{-n-sp},
\]
we get that 
\begin{eqnarray*} 
\nonumber &&
\sup_{x\in \text{supp}\,\varphi_j} \int_{\mathds{R}^n \setminus B_j} (\ell_j + (u(y))_-)^{p-1} \bar{K}(x,y) \,{\textrm d}y 
\\[1ex] \nonumber && \qquad\qquad \leq c 2^{j(n+sp)} \int_{\mathds{R}^n \setminus B_j} (\ell_j + (u(y))_-)^{p-1} |y-x_0|^{-n-sp}  \,{\textrm d}y 
\\[1ex] \nonumber && \qquad\qquad \leq c 2^{j(n+sp)} \ell_j^{p-1} r^{-sp} + c 2^{j(n+sp)}  \int_{\mathds{R}^n \setminus B_R} (u(y))_-^{p-1} |y-x_0|^{-n-sp}  \,{\textrm d}y 
\\[1ex] \nonumber && \qquad\qquad = c 2^{j(n+sp)} \ell_j^{p-1} r^{-sp} + c 2^{j(n+sp)} r^{-sp} \left(\frac{r}{R} \right)^{sp} [\text{\rm Tail}(u_-; {x_0}, R)]^{p-1} 
\\[1ex] \nonumber && \qquad\qquad \leq  c 2^{j(n+sp)} \ell_j^{p-1} r^{-sp}
,
\end{eqnarray*}
where we have also used the fact that $u$ is nonnegative in $B_R$, \eqref{tail cond exp pos} and $\delta k< \ell_j$.
Thus, in view of the three displays above,~\eqref{eq_17star} and~\eqref{eq_16star} yield
$$
(\ell_j-\ell_{j+1})^p \left( \frac{ |B_{j+1} \cap \{ u < \ell_{j+1}\}|}{|B_{j+1}|} \right)^{\frac{p}{p^\ast}}
\, \leq\, c\,2^{j(n +p +sp)} \ell_j^p  \frac{ |B_{j} \cap \{ u < \ell_{j}\}|}{|B_{j}|}.
$$
If we set
$$
\dys 
A_j := \frac{ |B_{j} \cap \{ u < \ell_{j}\}|}{|B_{j}|},
$$
then  the previous estimates can be read as follows
$$
A_{j+1}^{\frac{p}{p^\ast}} \, \leq \,c\,\frac{\ell_j^p\,2^{j(n +p +sp)}}{(\ell_j-\ell_{j+1})^p} A_j
\,\leq \, c \, 2^{j(n +2p +sp)} A_j
$$
that in turn implies
\begin{equation*} 
A_{j+1}
\, \leq \, c_1\, 2^{j\left(\frac{np^\ast}{p} +2p^\ast +sp^\ast\right)}\, A_j^{1+\frac{sp}{n-ps}},
\end{equation*}
where $c_1 \equiv c_1(n,s,p,\lambda,\Lambda)$. 
Now, we are ready to apply Lemma~\ref{lem_iter} with 
$$
c_0=c_1,\,\,\, b=2^{\frac{n p^\ast}{p} +2p^\ast +sp^\ast}>1\quad \text{and}\quad \beta=\frac{sp}{n-ps}>0,
$$
there.
One can check that, choosing $\delta$ small enough depending only on 
$n, s, p, \lambda, \Lambda, \nu$; 
i.~\!e.,
$$
\dys
0<\delta:= \frac{1}{4} \exp\left\{ -\frac{\bar c \,c_1^{\frac{n-sp}{ps}} 2^{\left(\frac np+s+2\right)\frac{n(n-ps)}{ps^2}}}{\sigma}\right\}<\frac 14,
$$
and applying~\eqref{eq:exp pos iter init} assure that 
\begin{equation*}\label{eq_a0}
\dys
A_0 \, \leq \, c_1^{-\frac{n-sp}{sp}}\, 2^{-\left(\frac np+s+2\right)\frac{n(n-ps)}{ps^2}},
\end{equation*}
then
$$
\dys
\lim_{j\to\infty} A_j = 0;
$$
that is $\dys \inf_{x\in B_{4r}} u(x) \geq \delta k$, from which the result follows easily.
\end{proof}

\vspace{2mm}

\section{Proof of the Nonlocal Harnack inequality}\label{sec_harnack} 

In this section, we prove the nonlocal Harnack inequality as given in Theorem~\ref{harnack}.  The idea is to combine in a suitable way the local boundedness given by Theorem~\ref{thm_local}, true for subsolutions, together with the expansion of positivity obtained in Section~\ref{sec_positivity} that allows us to prove the next estimate valid for the infimum of supersolutions to problem \eqref{problema}, by mean of the classical tools presented in Section~\ref{sec_classic} and taking into account the tail estimate for solutions given in forthcoming Lemma~\ref{lem_tail}. 

\begin{lemma}\label{est_inf}
Let $u\in W^{s,p}(\mathds{R}^n)$ be a weak supersolution to problem \eqref{problema} such that $u\geq 0$ in $B_R\equiv B_R({x_0})\subset\Omega$. Then there exist constants $\varepsilon\in(0,1)$ and $c\geq 1$, both depending only on $n,s,p, \lambda$ and $\Lambda$ such that
\begin{equation}\label{inf}
\left(\dashint_{B_r} u^\varepsilon\,{\rm d}x\right)^{\frac{1}{\varepsilon}}\leq c\inf_{B_r} u+c\left(\frac{r}{R}\right)^{\frac{sp}{p-1}}{\rm Tail} (u_-;x_0,R)
\end{equation}
whenever $B_r\equiv B_r(x_0)\subset B_R$, where ${\rm Tail}(\cdot)$ is defined in \eqref{def_tail}.
\end{lemma}

\begin{proof}
Let us define for any $t>0$
$$
A^i_t=\left\{x\in B_{r} : u(x)> t \,\delta^i -\frac{T}{1-\delta}\right\}, \quad i=0,1,2,...,
$$
where $\delta$ is given in Lemma~\ref{thm_exp2} and we denoted by~$T$ the following quantity
\begin{equation*}\label{def_t}
\dys 
T:= \left( \frac{r}R \right)^{\frac{sp}{p-1}}\ta(u_-; {x_0}, R).
\end{equation*}

We want to make use of the Krylov-Safonov covering Lemma~\ref{lem_KS} with $E=A_t^{i-1}$. Obviously we have $A_t^{i-1}\subset A_t^{i}$, for any $i=1,2,...$ Let $x\in B_{r} $ such that $B_{3\rho}(x)\cap B_{r} \subset [A_t^{i-1}]_{\bar{\delta}}$, it has, recalling the definition \eqref{E_delta},
$$
|A_t^{i-1}\cap B_{3\rho}(x)|> \bar{\delta}|B_\rho |=\frac{\bar{\delta}}{3^n}|B_{3\rho}|.
$$
We can now apply Lemma \ref{thm_exp2}, with $k=t \,\delta^{i-1} -\frac{T}{1-\delta}$ and $\sigma=\frac{\bar{\delta}}{3^n}$ there, to get
$$
u>\delta \left(t \,\delta^{i-1} -\frac{T}{1-\delta}\right)- T=t\delta^i-\frac{T}{1-\delta}\quad \text{in}\,\,\, B_{r},
$$  
and hence $[A_t^{i-1}]_{\bar{\delta}}\subset A_t^i$. By Lemma~\ref{lem_KS} we must either $A_t^i=B_{r}$ or $|A_t^i|\geq \frac{c_3}{\bar\delta} |A_t^{i-1}|$ for $c_3 \equiv c_3(n)$. In any case, we can deduce that if for some integer $m$ it holds
\begin{equation}\label{A0}
|A^0_t|>c_3\left(\frac{\bar\delta}{c_3}\right)^m |B_{r}|,
\end{equation} 
then
$$
|A^{m-1}_t|>c_3\bar\delta^{-1}|A_t^{m-2}|>\dots>c_3^{m-1}\bar \delta^{1-m}|A^0_t|>\bar\delta |B_{r}|
$$
and therefore $A_t^m=B_{r}$. This implies
$$
u>t\delta^m-\frac{T}{1-\delta}\quad \text{in}\,\,\, B_{r}.
$$
Now we can choose $m$ to be the smallest integer such that \eqref{A0} is satisfied, that is 
$$
m>\frac{1}{\log (\bar\delta/c_3)}\log\frac{|A^0_t|}{c_3|B_{r}|}.
$$ 
With this choice of $m$ we get 
\begin{equation*}\label{eq_est_inf}
\inf_{B_{r}} u > \delta\, t\left(\frac{|A^0_t|}{c_3|B_{r}|}\right)^{\!\frac{1}{\beta}}-\frac{ T}{1-\delta} \,, \qquad \beta := \frac{\log (\bar \delta/c_3)}{\log  \delta},
\end{equation*}
where now both $\delta$ and $\beta$ depend only on $n,s,p,\lambda,\Lambda$ and $c_3 \equiv c_3(n)$. Setting $\xi:=\inf_{B_{r}} u $ we get
\begin{equation}\label{acca}
\frac{\left|B_{r} \cap \left\{u>t-\frac{T}{1-\delta}\right\}\right|}{|B_{r}|} =\frac{|A^0_t|}{|B_{r}|} \leq c_3 \delta^{-\beta} \,t^{-\beta}\,
\left(\xi+\frac{T}{1-\delta}\right)^{\beta}.
\end{equation}
By Cavalieri's Principle, we have
$$
\dashint_{B_{r}} u^\varepsilon\, {\rm d} x
\, = \, 
 \varepsilon \int_0^\infty t^{\varepsilon-1} \frac{|B_{r} \cap \{u > t \}|}{|B_{r}|} \, {\rm d} t
$$
for any $\eps>0$.  Since
\[
\frac{|B_{r} \cap \{u > t \}|}{|B_{r}|} \leq 
\frac{\left|B_{r} \cap \left\{u>t-\frac{T}{1-\delta}\right\}\right|}{|B_{r}|} 
\]
and using the estimate in~\eqref{acca}, it holds
\begin{eqnarray}
\nonumber \eps \int_0^\infty t^{\varepsilon-1} \frac{|B_{r} \cap \{u > t \}|}{|B_{r}|} \, {\rm d} t &\leq & \eps \int_0^{a} t^{\varepsilon-1}\, {\rm d} t \\
\nonumber && + \eps \int_{a}^\infty t^{\varepsilon-1} c_3\, \delta^{-\beta} \,t^{-\beta}\left(\xi+\frac{T}{1-\delta}\right)^{\beta} \, {\rm d} t \\
\nonumber&\leq& a^\varepsilon +\eps
\, c_3\, \delta^{-\beta}\left(\xi+\frac{ T}{1-\delta}\right)^{\beta}\int_{a}^\infty t^{\varepsilon-1-\beta}\,{\rm d}t
\end{eqnarray}
for any $a>0$. In particular, taking $a := \xi+\frac{ T}{1-\delta}$ and $\varepsilon := \beta/2$, we finally get
\begin{equation*}\label{est_uq}
\dashint_{B_{r}} u^\varepsilon\, {\rm d} x \leq c \, \left(\xi +\frac{T}{1-\delta}\right)^\varepsilon.
\end{equation*}
This concludes the proof.
\end{proof}

The next lemma gives a precise control of the tail of the weak solutions.

\begin{lemma}\label{lem_tail}
Let $u\in W^{s,p}(\mathds{R}^n)$ be a weak solution to problem~\eqref{problema} such that $u\geq 0$ in $B_R({x_0})\subset \Omega$.
Then, for 
$0<r<R$,
\begin{equation}\label{eq_4star}
\dys
\ta(u_+; {x_0},r) \, \leq \, c \sup_{B_r} u + c \left(\frac{r}{R}\right)^{\frac{sp}{p-1}}\ta(u_-; {x_0},R),
\end{equation}
where ${\rm Tail}(\cdot)$ is defined in \eqref{def_tail} and the constants $c$ depend only on $n$, $p$, $s$, $\lambda$ and~$\Lambda$.

\end{lemma}
\begin{proof}
Set $k:=\sup_{B_r} u$ and take a smooth function $\varphi \in C^{\infty}_0(B_{r})$ such that $0\leq \varphi  \leq 1$, $\varphi \equiv 1$ in $B_{r/2}$ and $|D\varphi | \leq 8/r$; consider the following test fuction
$$
\eta := (u-2k)\varphi ^p.
$$
We have
\begin{eqnarray}\label{eq_i123}
0 & = & \int_{B_r}\int_{B_r}|u(x)-u(y)|^{p-2}(u(x)-u(y))(\eta(x)-\eta(y))\dn \nonumber \\[1ex]
&& + \int_{\mathds{R}^n \setminus B_r} \int_{B_r}|u(x)-u(y)|^{p-2}(u(x)-u(y))(u(x)-2k)\varphi ^p(x)\dn \nonumber \\[1ex]
&& - \int_{B_r} \int_{\mathds{R}^n \setminus B_r}|u(x)-u(y)|^{p-2}(u(x)-u(y))(u(y)-2k)\varphi ^p(y)\dn \nonumber \\[1ex]
& =: & I_1 +I_2 + I_3.
\end{eqnarray}

A first estimate on $I_2$ will suggest the following split
\begin{eqnarray}\label{eq_1star}
\dys
I_{2} & \geq & \int_{\mathds{R}^n \setminus B_r} \int_{B_r} k (u(y)-k)_+^{p-1}\varphi ^p(x) \dn \nonumber \\[1ex]
& & - \int_{\mathds{R}^n \setminus B_r}\int_{B_r} 2k \chi_{\{u(y)< k\}} (u(x)-u(y))^{p-1}_+ \varphi ^p(x)\dn\nonumber  \\[1ex]
& =: & I_{2,1}- I_{2,2}.
\end{eqnarray}

Now,
\begin{eqnarray}\label{eq_2star}
\dys 
I_{2,1} & \geq &ck \int_{\mathds{R}^n \setminus B_r} \int_{B_r} u_+(y)^{p-1}\varphi ^p(x) \dn \nonumber
- c k^p \int_{\mathds{R}^n \setminus B_r}\int_{B_r} \varphi ^p (x) \dn \\[1ex]
& \geq & ck|B_r| r^{-sp}[\ta(u_+; {x_0}, r)]^{p-1} - c k^p r^{-sp}|B_r|,
\end{eqnarray}
where we have also used the fact that $ \varphi \equiv 1$ in $B_{{r}/{2}}$ and $2|y-x_0|\geq |x-y|$.

Also,
\begin{eqnarray}\label{eq_2star2}
\dys
I_{2,2} & \leq & 2k\int_{B_R\setminus B_r} \int_{B_r} k^{p-1}\varphi ^p\dn + 2k\int_{\mathds{R}^n \setminus B_R} \int_{B_r} (k + u(y)_-)^{p-1}\varphi ^p(x) \dn \nonumber \\[1ex]
& \leq & c k^p r^{-sp}|B_r| + c k |B_r| R^{-sp} [\ta(u_-; {x_0}, R)]^{p-1}.
\end{eqnarray}

Observe that $I_3$ can be estimated in the same way. Thus, combining~\eqref{eq_1star} with~\eqref{eq_2star} and~\eqref{eq_2star2}, we get
\begin{eqnarray}\label{eq_3star}
\dys
 I_2 + I_3 & \geq & -c k^p r^{-sp} |B_r| - ck |B_r |R^{-sp}[ \ta(u_-; {x_0},R)]^{p-1} \\[1ex]
 && + \, c k |B_r| r^{-sp}[\ta(u_+; x_0,r)]^{p-1}. \nonumber
 \end{eqnarray}
 
 \vspace{1mm}
 
Now, it remains to estimate the contribution given by $I_1$. For this, assume $\varphi (x) \geq \varphi (y)$; the opposite being treated in the same way as below. For brevity, denote by $w:= (u-2k)$. 
For any $(x,y) \in B_r \times B_r$, we have
\begin{eqnarray}\label{eq_3star2}
&& |w(x)-w(y)|^{p-2}(w(x)-w(y))(w(x)\varphi ^p(x) - w(y)\varphi ^p(y) )\nonumber \\
&& \qquad \qquad \geq \ |w(x)-w(y)|^p\varphi ^p(x) 
- c|w(x)-w(y)|^{p-1}|w(y)|\varphi ^{p-1}(x)|\varphi (x)-\varphi (y)|  \nonumber \\[1ex]
&& \qquad \qquad \geq \frac{1}{2}|w(x)-w(y)|^p\varphi ^p(x) - c|w(y)|^p|\varphi (x)-\varphi (y)|^p \\[1ex]
&& \qquad \qquad \geq -c k^p |\varphi (x)-\varphi (y)|^p, \nonumber
\end{eqnarray}
where in~\eqref{eq_3star2} we used usual Young's Inequality and also the fact that $u\leq k$ in $B_r$.
From the estimates above, together with the fact the $\varphi$ is a smooth function, we can deduce
\begin{eqnarray}\label{eq_3star3}
\dys
I_1 & \geq & -c k^p \int_{B_r}\int_{B_r} |\varphi (x)-\varphi (y)|^p \dn \ \geq \ -c k^p r^{-p} \int_{B_r}\int_{B_r} |x-y|^{p-sp-n}\dxy \nonumber \\[1ex]
& \geq & -ck^p r^{-sp}|B_r|.
\end{eqnarray}

Collecting \eqref{eq_i123}, \eqref{eq_3star} and \eqref{eq_3star3}, we finally obtain
$$
\dys
\ta(u_+; {x_0}, r) \, \leq \, ck + c \left(\frac{r}{R}\right)^{\frac{sp}{p-1}}\ta(u_-; {x_0}, R),
$$
which is the desired result, recalling the definition of $k$.
\end{proof}

\vspace{2mm}

We are finally
ready to complete the proof of the nonlocal Harnack inequality as stated in Theorem~\ref{harnack}. 

\begin{proof}
[Proof of the nonlocal Harnack Inequality]

Set
\begin{equation*}\label{def_gamma}
\gamma:=\frac{(p-1)n}{sp^2}.
\end{equation*}
The estimate in~\eqref{sup_estimate} 
 yields
$$
\sup_{B_{\rho/2}} u \, \leq  \, c\,\delta \ta(u_+; {x_0}, \rho/2)+c\, \delta^{-\gamma} \left( \dashint_{B_\rho} u_+^p {\rm d}x\right)^{\frac{1}{p}},
$$
which, combined with~\eqref{eq_4star} becomes
$$
\sup_{B_{\rho/2}} u \, \leq \, c \delta^{-\gamma}\left( \dashint_{B_\rho} u_+^p {\rm d}x\right)^{\frac{1}{p}}
+ c \delta \sup_{B_\rho} u + c \delta \left(\frac{\rho}{R}\right)^{\frac{sp}{p-1}}\ta(u_-; {x_0}, R).
$$

Now, we want to apply the iteration Lemma~\ref{lem_gg}. For this, set $\rho=(\sigma-\sigma')r$ with $1/2 \leq \sigma' < \sigma \leq 1$. We have by a covering argument that
\begin{eqnarray*}\label{eq_5star}
\dys
\sup_{B_{\sigma'r}} u & \leq & c\frac{\delta^{-\gamma}}{(\sigma-\sigma')^{\frac{n}{p}}}\left( \dashint_{B_{\sigma r}} u^p \, {\rm d} x\right)^{\frac{1}{p}} + c \delta \sup_{B_{\sigma r}}u  +\, c \delta \left(\frac{r}{R}\right)^{\frac{sp}{p-1}}\ta(u_-; {x_0}, R) \nonumber \\[1ex]
& \leq &  c\frac{\delta^{-\gamma}}{(\sigma-\sigma')^{\frac{n}{p}}} \left(\sup_{B_{\sigma r} }u\right)^{\!\frac{p-s}{p}}\left( \dashint_{B_{\sigma r}} u^s \, {\rm d} x\right)^{\frac{1}{p}} + c \delta \sup_{B_{\sigma r}}u \nonumber \\[1ex]
&& +\, c \delta \left(\frac{r}{R}\right)^{\frac{sp}{p-1}}\ta(u_-; {x_0}, R).  \end{eqnarray*}
By choosing $\delta=1/(4c)$, a standard application of Young's Inequality \big(with exponent $p/s$ and $(p-s)/p$\big) yields
\begin{eqnarray*}\label{eq_5star2}
\dys
\sup_{B_{\sigma' r}} u & \leq & \frac{1}{2} \sup_{B_{\sigma r}} u + \frac{c}{(\sigma-\sigma')^{\frac{n}{q}}}
\left( \dashint_{B_r} u^q \, {\rm d} x \right)^{\frac{1}q} \\[1ex]
&& + \, c \left( \frac{r}R \right)^{\frac{sp}{p-1}}\ta(u_-; {x_0}, R), \nonumber \quad \forall \,\,q \in (0, p^\ast),
\end{eqnarray*}	
so that Lemma~\ref{lem_gg}, choosing in particular $f(t):=\sup_{B_{\sigma't}}u$, $\tau=\sigma r$, $t=\sigma' r$, $\theta=n/q$
there, gives
\begin{equation*}\label{eq_6star}
\dys
\sup_{B_r} u \, \leq \, c \left (\dashint_{B_{r}} u^q\,{\rm d}x\right)^{\frac{1}q} + c \left( \frac{r}R \right)^{\frac{sp}{p-1}}\ta(u_-; {x_0}, R), \quad \forall \,\, q \in (0, p^\ast),
\end{equation*}
where the constant $c$ depends also on $q$. To conclude the proof we combine the above estimate with that established in Lemma~\ref{est_inf}, setting $q=\varepsilon$. 
\end{proof}

\vspace{2mm}

\section{Proof of the Nonlocal weak Harnack inequality}\label{sec_weak} 

This section is devoted to the proof of a weak Harnack type inequality for both the superminima of the functional in~\eqref{funzionale} and the weak supersolutions to the problem~\eqref{problema}. Before starting with the proof of Theorem~\ref{thm_weak}, we want to prove a Caccioppoli type estimate given by the following
\begin{lemma}\label{cacio2}
Let $p\in(1,\infty)$, $q\in (1,p)$, $d>0$ and let $u\in W^{s,p}(\mathds{R}^n)$ be a weak supersolutions to problem \eqref{problema} such that $u\geq 0$ in $B_R(x_0)\subset\Omega$. Then, for any $B_r\equiv B_r(x_0)\subset B_{3R/4}(x_0)$ and any nonnegative $\varphi \in C^\infty_0(B_r)$, the following estimate holds true
\begin{eqnarray}\label{eq_cacio2}
&&\hspace{-0.5cm} \int_{B_r}\int_{B_r} K(x,y)|w(x)\varphi(x)-w(y)\varphi(y)|^p\dx\dy \nonumber\\
&&\qquad \quad \leq c\int_{B_r}\int_{B_r}\bar K(x,y)(\max\{w(x),w(y)\})^p |\varphi(x)-\varphi(y)|^p \dx\dy \\
&& \qquad \quad \quad +c \,\bigg( \sup_{z\in \text{\rm supp} \varphi} \int_{\mathds{R}^n \setminus B_r}\bar{K}(z,y)\dy  \nonumber
\\ &&\qquad \qquad \quad\quad +  d^{1-p}R^{-sp}
\left(\int_{B_r} w^p(x)\varphi^p(x)\dx \right) 
\big[\text{\rm Tail}(u_-; x_0, R)\big]^{p-1},
\nonumber
\end{eqnarray} 
where $w:= (u+d)^{\frac{p-q}{p}}$, $\bar{K}$ is defined in \eqref{bark} and the constants $c$ depend only on $p$ and $q$. 
\end{lemma}  
\begin{proof}
For any $d>0$, let $\tilde{u}:=u+d$ and let $\eta:=\tilde{u}^{1-q}\varphi^p$, where $\varphi \in C^{\infty}_0(B_{r})$, $r<R$, and $q \in [1+\sigma, p-\sigma]$; $\sigma>0$ small. 
Test the equation in~\eqref{weak_super} with $\eta$ as above, by recalling that $\tilde{u}$ is a weak supersolution, we have
\begin{eqnarray}\label{I123}
\nonumber 0 & \leq & 
\int_{B_{r}}\int_{B_{r}} |\tilde{u}(x)-\tilde{u}(y)|^{p-2}\big(\tilde{u}(x)-\tilde{u}(y)\big)\big(\eta(x)-\eta(y)\big) \dn \\
\nonumber & &  + \int_{\mathds{R}^n \setminus B_{ r}}\int_{B_{ r}} |\tilde{u}(x)-\tilde{u}(y)|^{p-2}\big(\tilde{u}(x)-\tilde{u}(y)\big)\eta(x) \dn \\
\nonumber &&  - \int_{B_r}\int_{\mathds{R}^n \setminus B_r} |\tilde{u}(x)-\tilde{u}(y)|^{p-2}\big(\tilde{u}(x)-\tilde{u}(y)\big)\eta(y) \dn \\
& =: & I_1+I_2+I_3. 
\end{eqnarray}
First of all, notice that for any $x\in B_R\supset B_r$ and for any $y\in \mathds{R}^n$,
$$
|\tilde{u}(x)-\tilde{u}(y)|^{p-2}\big(\tilde{u}(x)-\tilde{u}(y)\big) \leq c(\tilde{u}(x))^{p-1}+c(u(y))_-^{p-1} \ \
\text{and} \ \
\tilde{u}^{1-q}(x) \leq d^{1-p} \tilde{u}^{p-q}(x).
$$
Also, we have that $(u(y))_{-}$ vanishes for any $y\in B_R$, thanks to the assumptions on $u$.
Thus, the following estimate holds true.
\begin{eqnarray}\label{I23}
\nonumber I_2+I_3 & \leq & 
c\int_{\mathds{R}^n \setminus B_r} \int_{B_r}\tilde{u}^{p-1}(x) \eta(x) \dnb + c\int_{\mathds{R}^n \setminus B_r} \int_{B_r} \big(u(y)\big)_{-}^{p-1} \eta(x) \dnb \\
& \leq & c\, \bigg( \sup_{z\in \text{\rm supp} \phi} \int_{\mathds{R}^n \setminus B_r}\bar{K}(z,y)\dy +  d^{1-p}\!\int_{\mathds{R}^n \setminus B_R} \big(u(y)\big)_-^{p-1}|y-{x_0}|^{-n-sp}\dy \bigg)\\
\nonumber &&  \ \ \,
\times\!
\left(\int_{B_r} w^p(x) \varphi^p(x)\dx \right),
\end{eqnarray}
where we denoted by $w:=\tilde{u}^{\frac{p-q}{p}}$.
\vspace{1mm}

Now, we consider the integrand of $I_1$. In the case when $\tilde{u}(x)>\tilde{u}(y)$ we can use the following inequality, which is valid for any $\varepsilon\in (0,1)$,
$$
\varphi^p(x) \, \leq\, \varphi^p(y)+c_p\,\varepsilon\,\varphi^p(y)+(1+c_p\varepsilon)\,\varepsilon^{1-p}|\varphi(x)-\varphi(y)|^p,
$$
where $c_p=(p-1)\Gamma(\max\{1,p-2\})$; see Lemma 3.1 in~\cite{DKP13}. 
For any $\delta\in (0,1)$, we choose
$$
\varepsilon:=\delta\, \frac{\tilde{u}(x)-\tilde{u}(y)}{\tilde{u}(x)}\in (0,1)
$$ 
and we get
\begin{eqnarray*}
&& \hspace{-0.8cm}K(x,y) |\tilde{u}(x)-\tilde{u}(y)|^{p-2}(\tilde{u}(x)-\tilde{u}(y))\left[\frac{\varphi^p(x)}{\tilde{u}^{q-1}(x)}-\frac{\varphi^p(y)}{\tilde{u}^{q-1}(y)}\right] \\[0.8ex]
&&\leq  K(x,y) |\tilde{u}(x)-\tilde{u}(y)|^{p-2}\frac{(\tilde{u}(x)-\tilde{u}(y))}{\tilde{u}^{q-1}(x)}\,\varphi^p(y)\left[1+c_p\delta\frac{\tilde{u}(x)-\tilde{u}(y)}{\tilde{u}(x)}-\frac{\tilde{u}^{q-1}(x)}{\tilde{u}^{q-1}(y)} \right]\\
&& \quad + \,c\,K(x,y)\tilde{u}^{p-q}(x)\delta^{1-p}|\varphi(x)-\varphi(y)|^p.
\end{eqnarray*}
Note that the first term that appears in the  inequality above can be rewritten as follows
\begin{equation*}\label{def_j1o}
K(x,y) \frac{(\tilde{u}(x)-\tilde{u}(y))^p}{(\tilde{u}(x))^q}\,\varphi^p(y)\left(c_p\delta+\frac{1-\frac{\tilde{u}^{q-1}(x)}{\tilde{u}^{q-1}(y)}}{1-\frac{\tilde{u}(y)}{\tilde{u}(x)}}\right)=:J_1.
\end{equation*}
For this, consider the real function $t\mapsto g(t)$ given by 
$$
g(t):=\frac{1-t^{1-q}}{1-t}=-\frac{q-1}{1-t}\int_t^1\tau^{-q}\,{\rm d}\tau, \quad \forall\,\, t\in(0,1).
$$
Since $q>1$, $g(t)\leq -(q-1)$ for all $t\in(0,1)$. Moreover if $t\in (0,1/2]$ we have
$$
g(t)\leq -\,\frac{(q-1)}{2^{q}}\,\frac{t^{1-q}}{(1-t)}.
$$
Therefore, it is convenient to distinguish now the case when $2\tilde{u}(y)\leq\tilde{u}(x)$ and that when $2\tilde{u}(y)>\tilde{u}(x)$. In the first case,  take  $t=\tilde{u}(y)/\tilde{u}(x)\in (0,1/2]$, so that, in view of the considerations above, it yields
\begin{eqnarray}\label{eq_j11}
J_1 &\leq& K(x,y)\frac{(\tilde{u}(x)-\tilde{u}(y))^p}{(\tilde{u}(x))^q}\,\varphi^p(y)\left(c_p\delta - \frac{q-1}{2^{q}}\frac{\tilde{u}^{q-1}(x)}{\tilde{u}^{q-1}(y)}\frac{\tilde{u}(x)}{\tilde{u}(x)-\tilde{u}(y)}\right) \nonumber\\
&\leq& K(x,y)\frac{(\tilde{u}(x)-\tilde{u}(y))^{p-1}}{(\tilde{u}(y))^{q-1}}\,\varphi^p(y)\left(c_p\delta-\frac{q-1}{2^{q}}\right),
\end{eqnarray}
where we have also used that 
$$
\frac{(\tilde{u}(x)-\tilde{u}(y))\tilde{u}^{q-1}(y)}{(\tilde{u}(x))^q}\leq 1.
$$
Choosing $\delta$ as
\begin{equation}\label{def_delta1}
\delta=\frac{q-1}{2^{q+1} c_p}
\end{equation}
the equation in~\eqref{eq_j11} plainly yields
\begin{equation}\label{eq_j11b}
J_1 
\,\leq\, -K(x,y)\frac{q-1}{2^{q+1}}\frac{(\tilde{u}(x)-\tilde{u}(y))^{p-1}}{(\tilde{u}(y))^{q-1}}\varphi^p(y).
\end{equation}
Moreover, since we are assuming $2\tilde{u}(y)\leq \tilde{u}(x)$, we have
\begin{eqnarray*}
\frac{(\tilde{u}(x)-\tilde{u}(y))^{p-1}}{(\tilde{u}(y))^{q-1}}&\geq &\frac{2^{q-1}(\tilde{u}(x)-\tilde{u}(y))^{p-1}}{(\tilde{u}(x))^{q-1}}\\[1ex]
&\geq& 2^{q-p}(\tilde{u}(x))^{p-q}\\[1ex]
&\geq& 2^{q-p}\left((\tilde{u}(x))^{\frac{p-q}{p}}-(\tilde{u}(y))^{\frac{p-q}{p}}\right)^p.
\end{eqnarray*}
Finally, combining the preceding inequality together with~\eqref{eq_j11b}, we obtain the following estimate for $J_1$ in the case when $2\tilde{u}(y)\leq \tilde{u}(x)$,
\begin{equation}\label{J1_est1}
J_1\, \leq \, -\frac{q-1}{2^p} K(x,y)(w(x)-w(y))^p\varphi^p(y).
\end{equation}
\vspace{1mm}

It remains to consider the case when $2\tilde{u}(y)>\tilde{u}(x)$. Firstly, in view of the choice of the parameter~$\delta$ in~\eqref{def_delta1}, we have
\begin{eqnarray*}
J_1&\leq& K(x,y)\frac{(\tilde{u}(x)-\tilde{u}(y))^p}{(\tilde{u}(x))^q}\,\varphi^p(y)\big(c_p\delta-(q-1)\big)\\[1ex]
&=&-(2^{q+1}-1)\frac{q-1}{2^{q+1}} K(x,y) \frac{(\tilde{u}(x)-\tilde{u}(y))^p}{(\tilde{u}(x))^q}\,\varphi^p(y).
\end{eqnarray*}
Now, observe that
\begin{eqnarray*}
(w(x)-w(y))^p&=&\left(\frac{p-q}{p}\right)^p\left(\int_{\tilde{u}(y)}^{\tilde{u}(x)}t^{-\frac qp}\,{\rm d}t\right)^p\\[0.5ex]
&\leq&\left(\frac{p-q}{p}\right)^p\frac{1}{(\tilde{u}(y))^q}\left(\int_{\tilde{u}(y)}^{\tilde{u}(x)}{\rm d}t\right)^p\\[0.5ex]
&=&\left(\frac{p-q}{p}\right)^p\frac{(\tilde{u}(x)-\tilde{u}(y))^p}{(\tilde{u}(y))^q}\\[0.5ex]
&\leq& 2^q \left(\frac{p-q}{p}\right)^p\frac{(\tilde{u}(x)-\tilde{u}(y))^p}{(\tilde{u}(x))^q}.
\end{eqnarray*}
Hence, we get  
\begin{equation}\label{J1_est2}
J_1
\, \leq\, -(2^{q+1}-1)\frac{q-1}{2^{2q+1}}\left(\frac{p}{p-q}\right)^p K(x,y)(w(x)-w(y))^p\varphi^p(y).
\end{equation}
\vspace{1mm}

All in all, by comparing the estimates in~\eqref{J1_est1} and~\eqref{J1_est2}, we obtained the following estimate for the contribution in~$J_1$, when $\tilde{u}(x)>\tilde{u}(y)$, 
\begin{equation}\label{J1_est_fin}
J_1
\, \leq \, 
-c \,K(x,y)(w(x)-w(y))^p\varphi^p(y)
\end{equation}
where
$$
c=\min\left\{\frac{q-1}{2^{p+1}},\, (2^{q+1}-1)\frac{q-1}{2^{2q+1}}\left(\frac{p}{p-q}\right)^p\right\}.
$$

Observe that when $\tilde{u}(x) =\tilde{ u}(y)$, then the estimate~\eqref{J1_est_fin} trivially holds. 
\vspace{1mm}

On the other hand, in the case when $\tilde{u}(y) > \tilde{u}(x)$, it suffices just to exchange the roles of $x$ and $y$ in the whole computations made before. We finally arrive at  
\begin{eqnarray*}
\nonumber I_1 &\leq &-c\int_{B_r}\int_{B_r} |w(x)-w(y)|^p\,\varphi^p(y)\dn\\
&&+\,c\int_{B_r}\int_{B_r}(\max\{w(x),w(y)\})^p |\varphi(x)-\varphi(y)|^p \dnb.
\end{eqnarray*}

Now, putting the inequality above and \eqref{I23} in~\eqref{I123}, we obtain
\begin{eqnarray*}
&& \hspace{-0.5cm}\int_{B_r}\int_{B_r} |w(x)-w(y)|^p\,\varphi^p(y)\dn\\
&& \qquad \leq c\int_{B_r}\int_{B_r}(\max\{w(x),w(y)\})^p |\varphi(x)-\varphi(y)|^p \dnb \\
&& \qquad \quad +c \,\bigg( \sup_{z\in \text{\rm supp} \varphi} \int_{\mathds{R}^n \setminus B_r}\bar{K}(z,y)\dy 
 \\
&& \qquad \qquad \quad
+  d^{1-p}\int_{\mathds{R}^n \setminus B_R} \big(u(y)\big)_-^{p-1}|y-{x_0}|^{-n-sp}\dy \bigg)\left(\int_{B_r} w^p(x) \varphi^p(x)\dx \right),
\end{eqnarray*}
which, together with the fact that
\begin{eqnarray*}
|w(x)\varphi(x)-w(y)\varphi(y)|^p
\! & \leq & \! c\,\big(\max\big\{w(x),w(y)\big\}\big)^p |\varphi(x)-\varphi(y)|^p \\
& & \! +c\, |w(x)-w(y)|^p\,\varphi^p(y),
\end{eqnarray*}
yields the estimate~\eqref{eq_cacio2}.
\end{proof}

Now we are ready to prove Theorem~\ref{thm_weak} in a quite standard way.

\begin{proof}[Proof of Theorem~\ref{thm_weak}]
For simplicity of notation, we replace $r$ by $r/2$ below compared to the statement of the theorem. 
First, let $1/2 < \sigma'  < \sigma \leq 3/4$ and let $\varphi \in C_0^\infty(B_{\sigma r})$ be such that $\varphi = 1$ in $B_{\sigma' r}$ and $|D\varphi| \leq 4/[(\sigma-\sigma')r]$. We apply the fractional Sobolev inequality to the function $w\varphi$, with $w=\tilde{u}^{\frac{p-q}{p}}=(u+d)^{\frac{p-q}{p}}$. Together with the assumptions on the kernel $K$ we get
\begin{equation}\label{eq_fratio}
\left(\dashint_{B_{r}}|w(x)\varphi(x)|^{p^*}\,{\rm d}x\right)^{\frac{p}{p^*}}
 \leq \, c\, \frac{r^{sp}}{r^n}\int_{B_r}\int_{B_r}|w(x)\varphi(x)-w(y)\varphi(y)|^p \dn.
\end{equation}
Moreover, we have
\begin{eqnarray}\label{eqqq}
&& \int_{B_r}\int_{B_r}\big(\max\big\{w(x),w(y)\big\}\big)^p |\varphi(x)-\varphi(y)|^p \dnb \, \leq \, \frac{c\,r^{-sp}}{(\sigma-\sigma')^p}\int_{B_{\sigma r}}w^p(x)\,{\rm d}x,
\end{eqnarray}
where we also used the fact that $\varphi$ satisfies $|D\varphi|\leq 4/[(\sigma-\sigma')r]$.

Finally, by combining \eqref{eq_fratio} and \eqref{eqqq} with \eqref{eq_cacio2}, and recalling the definition in~\eqref{def_tail}, we get
\begin{eqnarray}\label{eq_dait}
&& \hspace{-0.5cm} \left(\dashint_{B_r}|w(x)\varphi(x)|^{p^*}\,{\rm d}x\right)^{\frac{p}{p^*}} \nonumber \\ 
&&\qquad\qquad  \leq \ c \left\{\frac1{(\sigma-\sigma')^p}+d^{1-p}\,\left(\frac r R\right)^{sp}\left[\text{\rm Tail}(u_-; {x_0}, R)\right]^{p-1}\right\}\dashint_{B_r} w^p(x)\,{\rm d}x,
\end{eqnarray}
where we also used that 
$$
 \sup_{z\in \text{\rm supp} \varphi} \int_{\mathds{R}^n \setminus B_r}\bar{K}(z,y)\dy 
 \leq c\,r^{-sp} .
$$
Choosing $d$ as in Lemma~\ref{thm_exp}, that is
\begin{equation}\label{d2}
d:=\frac 12\left(\frac{r}{R}\right)^{\frac{sp}{p-1}} \ta(u_-;x_0,R),
\end{equation}
and $\varphi\equiv 1$ in $B_{r/2}$ and recalling the definition of $w$, we deduce from \eqref{eq_dait}
$$
\left(\dashint_{B_{\sigma' r}} \tilde{u}^{(p-q)\frac{n}{n-sp}}\,\dx\right)^{\frac{n-sp}{n}}\leq\, \frac{c}{(\sigma-\sigma')^p}\,\dashint_{B_{\sigma r}} \tilde{u}^{p-q}\,\dx
$$
with $c=c(n,p,s,q,\lambda, \Lambda)$. Now if $q\in (1,p)$ by a standard finite Moser iteration (see, e.~\!g., Theorem~8.18 in~\cite{GT01} and also Theorem~1.2 in~\cite{Tru67}), we can get
\begin{equation}\label{finite_moser}
\left(\dashint_{B_{r}} \tilde{u}^{t}\,\dx\right)^{\frac{1}{t}}\leq c\left(\dashint_{B_{3r/4}} \tilde{u}^{t'}\,\dx\right)^{\frac{1}{t'}}\qquad \forall\,\,\, 0<t'<t<\frac{n(p-1)}{n-ps}.
\end{equation}
To get the desired result we have to apply Lemma~\ref{est_inf} to $\tilde{u}$, noticing that it is a weak supersolution to problem \eqref{problema}. Thus, by combining~\eqref{inf} with~\eqref{finite_moser} for $t'=\varepsilon$, we obtain the following estimate for $\tilde{u}$
$$
\left(\dashint_{B_{r/2}} \tilde{u}^{t}\,\dx\right)^{\frac{1}{t}}\leq c\inf_{B_{2r}} \tilde{u}+ c\left(\frac r R\right)^{\frac{sp}{p-1}}\ta(\tilde{u}_-;x_0,R).
$$
Finally, in order to arrive at~\eqref{eq1}, it is sufficient to notice that the following estimate holds
$$
\left(\dashint_{B_{r/2}}u^{t}\,\dx\right)^{\frac{1}{t}}\leq \left(\dashint_{B_{r/2}} \tilde{u}^{t}\,\dx\right)^{\frac{1}{t}},
$$
and to recall the definition of~$d$ given by~\eqref{d2}. The proof is complete.
\end{proof}

\vspace{2.5mm}

{\bf Acknowledgments.}
This paper was partially carried out while TK was visiting the University of Parma, 
supported by the \href{http://prmat.math.unipr.it/~rivista/eventi/2010/ERC-VP/}{ERC grant 207573 ``Vectorial Problems''}, and it was finalized during the program ``Evolutionary Problems'' at the Institut Mittag-Leffler (Djursholm, Sweden) in the Fall 2013.
GP has been also supported by PRIN~2010-11 ``Calcolo delle Variazioni''.
The authors gratefully acknowledge the support and the hospitality of all these institutions.

\vspace{3mm}

\end{document}